\let\ORIlabel\label
\let\ORIrefstepcounter\refstepcounter
\AddToHook{package/hyperref/before}{
   \let\label\ORIlabel
   \let\refstepcounter\ORIrefstepcounter}
\documentclass[hidelinks,onefignum,onetabnum]{siamart220329}

\usepackage{algpseudocodex}
\usepackage{color}
\usepackage{amssymb}
\usepackage{mathtools}
\usepackage{accents}

\ifpdf
  \DeclareGraphicsExtensions{.eps,.pdf,.png,.jpg}
\else
  \DeclareGraphicsExtensions{.eps}
\fi

\newlength{\dhatheight}

\newsiamremark{remark}{Remark}
\newsiamremark{hypothesis}{Hypothesis}
\crefname{hypothesis}{Hypothesis}{Hypotheses}
\newsiamthm{claim}{Claim}
\newsiamthm{lem}{Lemma}

\headers{General Superconvergence for Cubature on Polytopal Domains}{J. A. Reeger}

\title{A General Superconvergence Result for Cubature on Tesselated Polytopal Domains in Two and Three Dimensions\thanks{Submitted to the editors DATE.
\funding{This work was funded by the Air Force Office of Scientific Research project Kernel Methods with Machine Learning and Adaptivity.}}}

\author{Jonah A. Reeger\thanks{Department of Mathematics and Statistics, Air Force Institute of Technology, Wright-Patterson Air Force Base, OH
  (\email{jonah.reeger@afit.edu}). }}

\usepackage{amsopn}

\ifpdf
\hypersetup{
  pdftitle={An Example Article},
  pdfauthor={J. A. Reeger}
}
\fi

\begin{document}

\maketitle

% REQUIRED
\begin{abstract}
Cubature rules, which approximate definite integrals as a linear combination of a set of function values, are ubiquitous and necessary for computational methods in the physical sciences.  A superconvergence result for cubature rules on polytopal domains in two and three dimensions is developed, whereby a rule that is exact for all bivariate polynomials of a fixed even degree realizes an extra order of convergence under a decrease in the spacing between nodes.
\end{abstract}

% REQUIRED

% REQUIRED

\section{Introduction}

Inspired by observations in \cite{JAR2026}, this article presents a superconvergence result for cubature rules developed for approximating definite integrals on a polytopal domain, $\Omega\subset\mathbb{R}^{d}$, $d=2,3$, when the rules are exact for all $d$-variate polynomials up to a fixed degree.  It was observed for $d=2$ that such rules realize an extra order of convergence under refinement of the spacing between nodes, even when the node set does not feature any prescribed uniformity or configuration.  This result is a generalization of that presented in \cite{SymmetricTriangleQuadrature} and dispenses with the requirement that the set of cubature nodes and weights satisfy certain symmetries.

Cubature rules, which approximate definite integrals as a linear combination of a set of function values, are ubiquitous and necessary for computational methods in the physical science, e.g., the method of moments and finite-element method, that require the integration of an arbitrary function over a given domain \cite{TriangleQuadratureReview}.  To realize computational efficiency, the domain of integration is often subdivided into a set of subdomains, frequently $d$-simplices, and approximations of definite integrals over subdomains are summed together.  A recent survey of cubature rules for $d$-simplices does not appear to be available; however, many of the typical approaches for constructing sets of weights and nodes appear in \cite{LynessCools1994,Cools1997,TriangleQuadratureReview}.  Contemporary works have been more specific in their focus, leveraging symmetry or orthogonal polynomials, for instance, in their construction \cite{LiuLiu2024,WorkuHickenZingg2026,SymmetricTriangleQuadrature}.  Both the constraints of symmetry and the use of orthogonal polynomials often lead to requirements on specific node locations to ensure that a method is exact for as large a polynomial order as possible.  Still others approaches have focused on rules that apply to triangles and tetrahedra that extend to arbitrary bounded domains in up to three dimensions and with node sets that do not require a particular configuration \cite{JARBF2016,JARBFMLW2016,JARBF2017,JAR2020,JAR2022,JAR2024,JAR2026,ASMV06}.

A common thread among many of these methods is that they integrate exactly all polynomials up to a particular degree, $0\leq m\in\mathbb{Z}$.  As an analogy, considering Newton-Cotes cubature rules for an interval (see, e.g., \cite{Atkinson}), exactness for integrating polynomials over a simplex in 1-D translates to an extra order of convergence when the rule is constructed to be exact for polynomials of even degree. In the 1-D case, this extra order of convergence comes from being able to integrate all polynomials of degree $m+1$ (the next polynomial degree) exactly.  That is, the method achieves an extra degree of precision.  The remainder of this article shows that an extra order of convergence is achieved a different way when subdividing a $d$-dimensional polytopal domain into $d$-simplices, instead relying on favorable cancellations of the error.  The importance of achieving an extra order of convergence lies in the ability to achieve a similarly accurate result at a lower computational cost, since the expense often increases with the polynomial order for which a method is made exact.

The remainder of this work is organized as follows.  Section \ref{sec:ProblemStatement} introduces the problem mathematically.  Then section \ref{sec:dSimplexCubature} presents the approach to decomposing $d$-simplices into subdomains in two and three dimensions, and the cubature error on these subdomains, that enables the proof of the superconvergence result for polytopal domains in section \ref{sec:TotalError}.  Numerical experiments demonstrating these results are given in section \ref{sec:Numerics}.  Finally, section \ref{sec:Conclusions} provides some conclusions.

\section{Problem Statement} \label{sec:ProblemStatement}

The superconvergence result developed here assumes that the cubature rule under consideration relies first on subdividing $\Omega\subset\mathbb{R}^{d}$ into a set of subdomains that are each a $d$-simplex.  That is, $\Omega=\cup_{k=1}^{K}t_{k,0}$, with $t_{k,0}$ the $d$-simplex that is the convex hull of the set of its $d+1$ vertices $\mathbf{v}_{k,l}$, $l=1,2,\ldots,d$, and with volume $h_{k}^{d}$ and barycenter $\mathbf{b}_{k}=(1/(d+1))\sum_{l=1}^{d+1}\mathbf{v}_{k,l}$.  Explicitly, $t_{k,0}=\mbox{conv}(\{\mathbf{v}_{k,l}\}_{l=1}^{d+1})$, with $\mbox{conv}$ representing the convex hull of the set.  Supposing further that $t_{k,0}$ and $t_{k',0}$ intersect in at most a common $d-1$-simplex when $k\neq k'$, properties of integration allow
\begin{align}
    \int\limits_{\Omega}f(\mathbf{x})dV=\sum\limits_{k=1}^{K}\int\limits_{t_{k,0}}f(\mathbf{x})dV.\nonumber
\end{align}
After subdivision the definite integral on $t_{k,0}$ is approximated using a set of $n_{k}$ local cubature nodes $\mathcal{X}_{k,0}=\{\mathbf{x}_{k,j}\}_{j=1}^{n_{k}}$ and weights $\{w_{k,j}\}_{j=1}^{n_{k}}$ by way of
\begin{align}
    \int\limits_{t_{k,0}}f(\mathbf{x})dV=\sum\limits_{j=1}^{n_{k}}w_{k,j}f(\mathbf{x}_{k,j})+\mathcal{E}_{t_{k,0},h_{k}}[f].\label{eq:TriangleRule}
\end{align}
It is ideal that the error $\mathcal{E}_{t_{k,0},h_{k}}[f]\to0$ as $h_{k}\to 0$ for all $f$ that are sufficiently many times differentiable.  In particular, methods are most often constructed so that $\mathcal{E}_{t_{k,0},h_{k}}[f]=O(h_{k}^{\rho})$ as $h_{k}\to 0$, for $\rho\geq1$.  To achieve this geometric convergence with respect to $h_{k}$, the set of cubature weights and nodes is often chosen so that the associated rule is exact for all polynomials up to a certain degree, allowing cancellation of terms in the Taylor formula of a smooth enough function (as described in section \ref{sec:ErrorTriangle}). Consider the set of $M_{d,m}=\left(\begin{array}{c}m+d\\ d\end{array}\right)$ unique multiindices $\boldsymbol{\alpha}\in\mathbb{Z}_{+}^{d}$ that have order at most $m$.  It is sometimes convenient to enumerate these multiindices $\boldsymbol{\alpha}_{l}$, $l=1,2,\ldots,M_{d,m}$. The set $\{\pi_{k,\boldsymbol{\alpha}}\}_{\lvert\boldsymbol{\alpha}\rvert\leq m}$ with, e.g., $\pi_{k,\boldsymbol{\alpha}}(\mathbf{x})=(\mathbf{x}-\mathbf{b}_{k})^{\boldsymbol{\alpha}}$, forms a basis for the space, $\mathbb{P}_{m}^{d}$, of $d$-variate polynomials up to degree $m$. Exactness for polynomials up to degree $m$ requires $\mathcal{E}_{t_{k,0},h_{k}}[p]=0$ for all polynomials $p\in\mathbb{P}_{m}^{d}$.

Approximating a definite integral over a polytopal domain by summing approximations over the $d$-simplex subdomains amounts to
\begin{align}
    \int\limits_{\Omega}f(\mathbf{x})dV=\sum\limits_{k=1}^{K}\sum\limits_{j=1}^{n_{k}}w_{k,j}f(\mathbf{x}_{k,j})+\mathcal{E}_{\Omega,h}[f],\label{eq:PolytopalCubature}
\end{align}
where $\mathcal{E}_{\Omega,h}[f]=\sum_{k=1}^{K}\mathcal{E}_{t_{k,0},h_{k}}[f]$, with $h=\max\{h_{k}\}_{k=1}^{K}$.  Section \ref{sec:TotalError} shows that under a mild assumption relating $K$ to $h$ a superconvergence result for the polytopal domain $\Omega$ can be achieved.

\section{Development of the Superconvergence Result for $d$-Simplex Subdomains} \label{sec:dSimplexCubature}

Development of the superconvergence result for the method \eqref{eq:PolytopalCubature} first requires a result for the $d$-simplex subdomains.  This section presents the superconvergence result for individual $d$-simplices under a convenient choice of refinement, while also presenting a typical, generic approach to constructing weight sets that are exact for all polynomials up to a certain degree given a set of cubature nodes.

\subsection{Error in Approximating Definite Integrals on a $d$-Simplex} \label{sec:ErrorTriangle}

We begin with a useful theorem that na\"{i}vely predicts a pessimistic convergence order on a polytopal domain achieved by summing cubature rules on $d$-simplex subdomains that are exact for even order polynomials.
\begin{theorem} \label{thm:LocalErrorOrder}
Suppose that
\begin{enumerate}
\item $\{w_{k,j}\}_{j=1}^{n}$ is a set of cubature weights for approximating the definite integral of a function over the $d$-simplex, $t_{k,0}$, and that $\mathcal{X}_{k,0}$ is the set of corresponding cubature nodes
\item $f\vcentcolon\mathbb{R}^{d}\to\mathbb{R}$ has continuous mixed partial derivatives up to order $m+2$ in an open convex $\mathcal{S}_{k,0}$ containing $\mathcal{X}_{k,0}\bigcup\{\mathbf{v}_{k,i}\}_{i=1}^{3}$,
\item the mixed partial derivatives are bounded, i.e., $\lvert\partial^{\boldsymbol{\alpha}}f(\mathbf{x})\rvert\leq B_{\boldsymbol{\alpha}}[f]$ for each $\mathbf{x}\in\mathcal{S}_{k,0}$ when $\lvert\boldsymbol{\alpha}\rvert\leq m+2$,
\item the cubature weights satisfy $w_{k,j}=O(h_{k}^{d})$ as $h_{k}\to0$ and
\item the cubature rule is exact for all polynomials in $\mathbb{P}_{m}^{d}$,
\end{enumerate}
then as $h_{k}\to0$
\begin{align}
\mathcal{E}_{t_{k,0},h_{k}}[f]=&\sum_{\lvert\boldsymbol{\alpha}\rvert= m+1}\frac{1}{\boldsymbol{\alpha} !} \partial^{\boldsymbol{\alpha}}f(\mathbf{x})\big|_{\mathbf{x}=\mathbf{b}_{k}}\mathcal{E}_{t_{k,0},h_{k}}[\pi_{k,\boldsymbol{\alpha}}]+O(h_{k}^{m+2+d}).\label{eq:nominalerror}
\end{align}
\end{theorem}

\begin{proof}
Under assumption 2, $f$ can be written as \cite[Theorem 2.68]{Folland2002AdvancedCalculus}
\begin{align}
    f(\mathbf{x}) =  \sum\limits_{\lvert\boldsymbol{\alpha}\rvert\leq m}\frac{1}{\boldsymbol{\alpha} !} \partial^{\boldsymbol{\alpha}}f(\mathbf{x})\big|_{\mathbf{x}=\mathbf{b}_{k}}\pi_{k,\boldsymbol{\alpha}}(\mathbf{x})+\sum_{\lvert\boldsymbol{\alpha}\rvert= m+1}\frac{1}{\boldsymbol{\alpha} !} \partial^{\boldsymbol{\alpha}}f(\mathbf{x})\big|_{\mathbf{x}=\mathbf{b}_{k}}\pi_{k,\boldsymbol{\alpha}}(\mathbf{x})+(\mathcal{R}_{m+1}f)(\mathbf{x}), \label{eq:expanded_taylor}
\end{align}
with remainder term
\begin{align}
(\mathcal{R}_{m+1}f)(\mathbf{x})=\sum\limits_{\lvert\boldsymbol{\alpha}\rvert= m+2}\frac{m+2}{\boldsymbol{\alpha}!}\left(\int\limits_{0}^{1}\partial^{\boldsymbol{\alpha}}f(\mathbf{x}')|_{\mathbf{x}'=\mathbf{b}_{k}+\tau(\mathbf{x}-\mathbf{b}_{k})}(1-\tau)^{|\boldsymbol{\alpha}|-1}d\tau\right)\pi_{k,\boldsymbol{\alpha}}(\mathbf{x}).\nonumber
\end{align}
at each point in $\mathcal{S}_{k,0}$, particularly at the nodes in $\mathcal{X}_{k,0}$. Performing the change of variables $\mathbf{x}\mapsto\mathbf{b}_{k}+h_{k}\boldsymbol{\beta}$, where $\boldsymbol{\beta}$ encodes the relative distance and direction from $\mathbf{b}_{k}$ to $\mathbf{x}$, assumptions 3 and 4 imply
\begin{align}
\left\lvert\int_{t_{k,0}}(\mathcal{R}_{m+1}f)(\mathbf{x})dV\right\rvert\leq\left(\sum\limits_{\lvert\boldsymbol{\alpha}\rvert= m+2}\frac{m+2}{\boldsymbol{\alpha}!}B_{\boldsymbol{\alpha}}[f]\int_{\frac{1}{h_{k}}(t_{k,0}-\mathbf{b}_{k})}\boldsymbol{\beta}^{\boldsymbol{\alpha}}dV\right)h_{k}^{m+2+d}\label{eq:TaylorRemainderIntegralBound}
\end{align}
(an analogous result for the remainder term itself appears in \cite[Corollary 2.75]{Folland2002AdvancedCalculus} and also for integrals in \cite[Section 4]{JAR2026}).  Similarly, writing $\mathbf{x}_{k,j}=\mathbf{b}_{k}+h_{k}\boldsymbol{\beta}_{k,j}$
\begin{align}
    \left\lvert\sum_{j=1}^{n_{k}}w_{k,j}(\mathcal{R}_{m+1}f)(\mathbf{x}_{k,j})\right\rvert\leq h_{k}^{m+2}\sum\limits_{\lvert\boldsymbol{\alpha}\rvert}\left(\sum_{j=1}^{n_{k}}\lvert w_{k,j}\rvert B_{\boldsymbol{\alpha}}[f]\boldsymbol{\beta}_{k,j}^{\boldsymbol{\alpha}}\right)\label{eq:TaylorRemainderCubatureBound}
\end{align}
Appealing to assumption 4 and combining \eqref{eq:TaylorRemainderIntegralBound} and \eqref{eq:TaylorRemainderCubatureBound}, as $h_{k}\to0$
\begin{align}
\int_{t_{k,0}}(\mathcal{R}_{m+1}f)(\mathbf{x})dV-\sum_{j=1}^{n_{k}}w_{k,j}(\mathcal{R}_{m+1}f)(\mathbf{x}_{k,j})=O(h_{k})^{m+2+d}.
\end{align}
Assumption 5 and substitution of \eqref{eq:expanded_taylor} into
\begin{align}
\mathcal{E}_{t_{k,0},h_{k}}[f]=\int_{t_{k,0}}f(\mathbf{x})dV-\sum_{j=1}^{n_{k}}w_{k,j}f(\mathbf{x}_{k,j})
\end{align}
produces the desired result.
\end{proof}

Similar arguments can be used to show that $\mathcal{E}_{t_{k,0},h_{k}}[\pi_{k,\boldsymbol{\alpha}}]=O(h_{k}^{m+1+d})$, so that in the general case $\mathcal{E}_{t_{k,0},h_{k}}[f]=O(h_{k}^{m+1+d})$, both as $h_{k}\to0$.  Further, when $K=O(h^{-d})$ as $h\to0$ 
\begin{align}
\mathcal{E}_{\Omega,h}[f]=\sum_{k=1}^{K}\mathcal{E}_{t_{k,0},h_{k}}[f]=O(h^{m+1}).\nonumber
\end{align}
The relationship between $h$ and $K$ is certainly the case if, e.g., Delaunay tesselations of quasi-uniformly spaced node sets are considered (e.g., \cite{HW2005}, Proposition 14.1), and is even a fair assumption triangulations of randomly spaced nodes that are drawn from distributions that feature a fill distance that scales appropriately and the triangulation is shape regular \cite{HW2005,Ciarlet2002FiniteElement}. The remainder of this work demonstrates that when $m$ is even, an extra order of convergence is realized.  First, however, is a discussion on construction of weight sets that are exact for polynomials up to a certain degree that satisfy assumption 4 of the theorem.

\subsection{Construction of Weight Sets Exact for Polynomials Up to a Fixed Degree} \label{sec:WeightSets}

Supposing that the set of cubature nodes is $\mathbb{P}_{m}^{d}$ unisolvent, the $n_{k}\times M_{d,m}$ matrix $P_{k}$, with entries $[P_{k}]_{jl}=\pi_{k,\boldsymbol{\alpha}_{l}}(\mathbf{x}_{k,j})$, has full rank (this requires that $n_{k}\geq M_{d,m}$) (\cite{HW2005}, Definition 2.6).  Cubature weights that are exact for all polynomials up to degree $m$ must satisfy the system of linear equations
\begin{align}
P_{k}^{T}\mathbf{w}_{k}=\boldsymbol{\pi}_{k}\label{eq:WeightSystem}
\end{align}
with $[\mathbf{w}_{k}]_{j} = w_{k,j}$, $j=1,2,\ldots,n_{k}$, and $[\boldsymbol{\pi}_{k}]_{l} = \int_{t_{k,0}}\pi_{k,\boldsymbol{\alpha}_{l}}(\mathbf{x})dV$, $l=1,2,\ldots,M_{d,m}$.  There are clearly many solutions of this system of equations when $n_{k}>M_{d,m}$, and often a solution that minimizes an objective, $J_{k}\vcentcolon\mathbb{R}^{{n}_{k}}\to\mathbb{R}$ is sought.  For instance, the solution that minimizes $J_{k}(\mathbf{w}_{k})=\lVert\mathbf{w}_{k}\rVert_{2}^{2}$ is common.  Alternatively, in the context of generating cubature nodes by integrating an interpolating function that is a linear combination of conditionally positive definite radial basis functions, $\varphi\vcentcolon\mathbb{R}_{+}\to\mathbb{R}$, and polynomials, an optimizer (or at least a stationary point) for the quadratic program with objective
\begin{align}
J_{k}(\mathbf{w}_{k})=\frac{1}{2}\mathbf{w}_{k}^{T}\Phi_{k}\mathbf{w}_{k}-\mathbf{w}_{k}^{T}\boldsymbol{\phi}_{k}\label{eq:RBFObjective}
\end{align}
and constraints \eqref{eq:WeightSystem} is determined \cite{VB2019}.  In \eqref{eq:RBFObjective} $\Phi_{k}$ is an $n_{k}\times n_{k}$ matrix with entries $[\Phi_{k}]_{ij}=\varphi(\lVert\mathbf{x}_{k,i}-\mathbf{x}_{k,j}\rVert_{2})$ and $\boldsymbol{\phi}_{k}$ is an $n_{k}\times 1$ vector with entries $[\boldsymbol{\phi}_{k}]_{j} = \int_{t_{k,0}}\varphi(\lVert\mathbf{x}-\mathbf{x}_{k,j}\rVert_{2})dV$, $i,j=1,2,\ldots,n_{k}$.

Writing again $\mathbf{x}_{k,j}=\mathbf{b}_{k}+h_{k}\boldsymbol{\beta}_{k,j}$, implies $[P_{k}]_{j,l}=h_{k}^{\lvert\boldsymbol{\alpha}_{l}\rvert}\boldsymbol{\beta}_{k,j}^{\boldsymbol{\alpha}_{l}}$.  A similar change of variables demonstrates that $[\boldsymbol{\pi}_{k}]_{l} = O(h_{k}^{\lvert\boldsymbol{\alpha}_{l}\rvert+d})$.  Therefore, the solution to the system of linear equations $P_{k}^{T}\mathbf{w}_{k}=\boldsymbol{\pi}_{k}$ has entries that satisfy $[\mathbf{w}_{k}]_{j}=O(h_{k}^{d})$ as $h_{k}\to0$, $j=1,2,\ldots,n_{k}$ (see, e.g., \cite{DavydovSchaback} or \cite{JAR2026} for more detailed discussions of similar results).

\subsection{Edge Midpoint Subdivisions of Simplices}

The superconvergence result relies on edge midpoint subdivisions of the simplices for $d=2$ (triangles) and $d=3$ (tetrahedra).  Further results on such subdivisions (even in dimension greater than 3) can be found in, e.g., \cite{EdgewiseSubdivision}.

\subsubsection{Congruent Subdivision of a Triangle} \label{sec:CongruentSubdivision}

When $d=2$, the triangle $t_{k,0}$ can be further subdivided into four congruent triangles by connecting the midpoints of its sides with three line segments. Figure \ref{fig:CongruentTriangleSubdivision} illustrates an example of the subdivision of a triangle, $t_{k,0}$, into four congruent triangles, each with the same interior angles as $t_{k,0}$.  The four congruent triangles are $t_{k,i}=\mbox{conv}(\{(1/2)(\mathbf{v}_{k,l}+\mathbf{v}_{k,i})\}_{l=1}^{3})$, $i=1,2,3$, and $t_{k,4}=\mbox{conv}(\{-(1/2)(\mathbf{v}_{k,l}-\mathbf{b}_{k})+\mathbf{b}_{k}\}_{l=1}^{3})$, each with volume (area) $(h_{k}/2)^{2}$.  Now, for each of these triangles define the affine transformations
\begin{align}
    T_{k,i}(\mathbf{x})&= (1/2)(\mathbf{x}+\mathbf{v}_{k,i}),\mbox{ }i=1,2,3, and\nonumber\\
    T_{k,4}(\mathbf{x})&= -(1/2)(\mathbf{x}-\mathbf{b}_{k})+\mathbf{b}_{k}
\end{align}
that satisfy $T_{k,i}(t_{k,0}) = t_{k,i}$.  These transformations (and similar ones in the next section for $d=3$) then allow that
\begin{align}
\int\limits_{t_{k,i}}f(\mathbf{x})dV = \left(\frac{1}{2}\right)^{d}\int\limits_{t_{k,0}}f(T_{k,i}(\mathbf{x}))dV,\nonumber
\end{align}
and they will be used in section \ref{sec:TotalError} to relate the errors when integrating over $t_{k,0}$ to the error realized by approximating integrals over each of $t_{k,i}$, $i=1,2,3,4$, then summing the results.

\begin{figure}
\begin{center}
\includegraphics[width=0.3\linewidth]{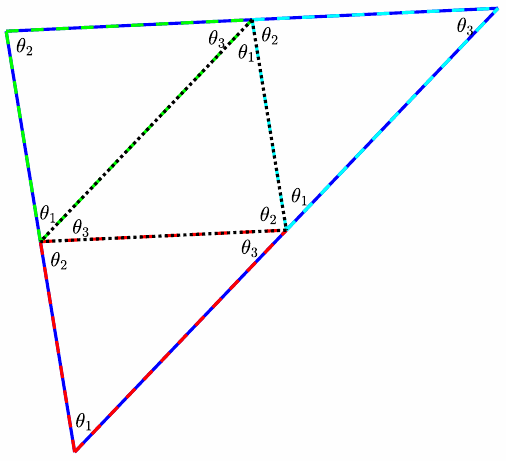}
\end{center}
\caption{A triangle $t_{k,0}$ with edges represented by solid (blue) line segments overlayed with the four congruent triangles, $t_{k,i}$, $i=1,2,3$, with edges represented by dashed line segments (cyan, green and red, where color is available) and $t_{k,4}$ with edges represented by black dotted line segments.  Equal interior angles are labeled $\theta_{i}$, $i=1,2,3$.}
\label{fig:CongruentTriangleSubdivision}
\end{figure}

\subsubsection{Edgewise Subdivision of a Tetrahedron} \label{sec:EdgewiseSubdivision}

In analogy to section \ref{sec:CongruentSubdivision}, when $d=3$ a tetrahedron $t_{k,0}$ is subdivided into 8 tetrahedral subdomains, each with volume $(h_{k}/2)^{3}$.  The first four of the eight tetrahedral subdomains include a vertex of $t_{k,0}$ and are given by
\begin{align}
t_{k,1} = & \mbox{conv}\{\mathbf{v}_{k,1},(1/2)(\mathbf{v}_{k,1}+\mathbf{v}_{k,2}),(1/2)(\mathbf{v}_{k,1}+\mathbf{v}_{k,3}),(1/2)(\mathbf{v}_{k,1}+\mathbf{v}_{k,4}),\nonumber\\ t_{k,2} = & \mbox{conv}\{\mathbf{v}_{k,2},(1/2)(\mathbf{v}_{k,1}+\mathbf{v}_{k,2}),(1/2)(\mathbf{v}_{k,2}+\mathbf{v}_{k,3}),(1/2)(\mathbf{v}_{k,2}+\mathbf{v}_{k,4}),\nonumber\\
t_{k,3} = & \mbox{conv}\{\mathbf{v}_{k,3},(1/2)(\mathbf{v}_{k,1}+\mathbf{v}_{k,3}),(1/2)(\mathbf{v}_{k,2}+\mathbf{v}_{k,3}),(1/2)(\mathbf{v}_{k,3}+\mathbf{v}_{k,4}),\mbox{ and}\nonumber\\
t_{k,4} = & \mbox{conv}\{\mathbf{v}_{k,4},(1/2)(\mathbf{v}_{k,1}+\mathbf{v}_{k,4}),(1/2)(\mathbf{v}_{k,2}+\mathbf{v}_{k,4}),(1/2)(\mathbf{v}_{k,3}+\mathbf{v}_{k,4}).\nonumber \end{align}
The remaining central tetrehedral subdomains have as their vertices only edge midpoints and can be represented as
\begin{align}
t_{k,5} = & \mbox{conv}\{(1/2)(\mathbf{v}_{k,3}+\mathbf{v}_{k,4}),(1/2)(\mathbf{v}_{k,1}+\mathbf{v}_{k,2}),(1/2)(\mathbf{v}_{k,1}+\mathbf{v}_{k,3}),(1/2)(\mathbf{v}_{k,1}+\mathbf{v}_{k,4}),\nonumber\\ t_{k,6} = & \mbox{conv}\{(1/2)(\mathbf{v}_{k,1}+\mathbf{v}_{k,2}),(1/2)(\mathbf{v}_{k,3}+\mathbf{v}_{k,4}),(1/2)(\mathbf{v}_{k,1}+\mathbf{v}_{k,4}),(1/2)(\mathbf{v}_{k,2}+\mathbf{v}_{k,4}),\nonumber\\
t_{k,7} = & \mbox{conv}\{(1/2)(\mathbf{v}_{k,3}+\mathbf{v}_{k,4}),(1/2)(\mathbf{v}_{k,1}+\mathbf{v}_{k,2}),(1/2)(\mathbf{v}_{k,2}+\mathbf{v}_{k,3}),(1/2)(\mathbf{v}_{k,1}+\mathbf{v}_{k,3}),\mbox{ and}\nonumber\\
t_{k,8} = & \mbox{conv}\{(1/2)(\mathbf{v}_{k,1}+\mathbf{v}_{k,2}),(1/2)(\mathbf{v}_{k,3}+\mathbf{v}_{k,4}),(1/2)(\mathbf{v}_{k,2}+\mathbf{v}_{k,4}),(1/2)(\mathbf{v}_{k,2}+\mathbf{v}_{k,3}).\nonumber
\end{align}
Figure \ref{fig:TetrahedronSubdivision} illustrates an example of the subdivision of a tetrehdron, $t_{k,0}$, into the eight tetrahedral subdomains.

\begin{figure}
\begin{center}
\includegraphics[width=\linewidth]{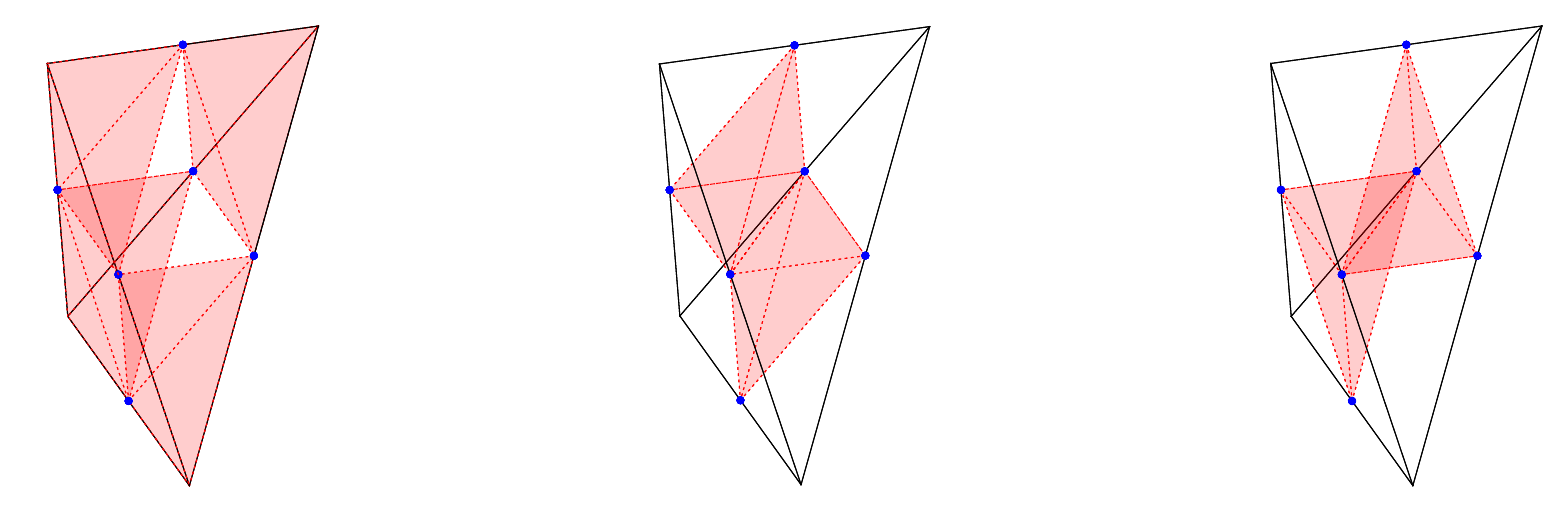}
\end{center}
\caption{A tetrahedron $t_{k,0}$ with edges represented by solid (black) line segments overlayed with the (left) four congruent tetrahedral subdomains and (center/right) 2 congruent tetrahedral subdomains.  In each frame, subdomain edges are represented by dotted line segments (red where color is available). The blue dots indicate the midpoints of the edges of $t_{k,0}$, which form the vertices of the tetrahedral subdomains that are not vertices of $t_{k,0}$.}
\label{fig:TetrahedronSubdivision}
\end{figure}

Define the $3\times3$ matrices $E = \left[\begin{array}{ccc}\mathbf{v}_{k,2}-\mathbf{v}_{k,1} & \mathbf{v}_{k,3}-\mathbf{v}_{k,1} & \mathbf{v}_{k,4}-\mathbf{v}_{k,1}\end{array}\right]$, $P=\left[\begin{array}{ccc}-\mathbf{e}_{1}+\mathbf{e}_{2}+\mathbf{e}_{3} & \mathbf{e}_{3}-\mathbf{e}_{1} & \mathbf{e}_{2}-\mathbf{e}_{1}\end{array}\right]$ and $Q=\left[\begin{array}{ccc}-\mathbf{e}_{1}+\mathbf{e}_{2}+\mathbf{e}_{3} & \mathbf{e}_{3} & \mathbf{e}_{2}\end{array}\right]$.  Further consider the affine transformations
\begin{align}
 T_{k,i}(\mathbf{x}) &= (1/2)(\mathbf{x}+\mathbf{v}_{k,i}), \mbox{ }i=1,2,3,4, \nonumber\\
 T_{k,5}(\mathbf{x}) &= (1/2)(\mathbf{v}_{k,1}+\mathbf{v}_{k,2})+(1/2)EPE^{-1}(\mathbf{x}-\mathbf{v}_{k,1}),\nonumber\\
 T_{k,6}(\mathbf{x}) &= -((1/2)(\mathbf{v}_{k,1}+\mathbf{v}_{k,2})-\mathbf{b}_{k})+\mathbf{b}_{k}-(1/2)EQE^{-1}(\mathbf{x}_{k,j}-\mathbf{v}_{k,1}),\nonumber\\
 T_{k,7}(\mathbf{x}) &= (1/2)(\mathbf{v}_{k,1}+\mathbf{v}_{k,2})+(1/2)EQE^{-1}(\mathbf{x}-\mathbf{v}_{k,1}),\mbox{ and}\nonumber\\
 T_{k,8}(\mathbf{x}) &= -((1/2)(\mathbf{v}_{k,1}+\mathbf{v}_{k,2})-\mathbf{b}_{k})+\mathbf{b}_{k}-(1/2)EPE^{-1}(\mathbf{x}-\mathbf{v}_{k,1}),\nonumber
\end{align}
such that $T_{k,i}(t_{k,0})=t_{k,i}$.  Notice that $T_{k,6}(\mathbf{x})-\mathbf{b}_{k}=-(T_{k,7}(\mathbf{x})-\mathbf{b}_{k})$ and  $T_{k,8}(\mathbf{x})-\mathbf{b}_{k}=-(T_{k,5}(\mathbf{x})-\mathbf{b}_{k})$. Further, $T_{k,5}(\mathbf{x})-\mathbf{b}_{k}=(1/2)EPE^{-1}(\mathbf{x}-(1/2)(\mathbf{v}_{k,1}+\mathbf{v}_{k,2}))$ and $T_{k,7}(\mathbf{x})-\mathbf{b}_{k}=(1/2)EQE^{-1}(\mathbf{x}-(1/2)(\mathbf{v}_{k,1}+\mathbf{v}_{k,2}))$. These relationships are useful in section \ref{sec:TotalError} to relate the error when approximating an integral over $t_{k,0}$ to the sum of the errors over the subdomains.

\subsection{Approximating Integrals Over Subdomains}

Now suppose that the integral of $f$ is approximated over each of the subdomains $t_{k,i}$, $i=1,2,\ldots,2^{d}$, but using the sets of nodes $\mathcal{X}_{k,i}=\{T_{k,i}(\mathbf{x}_{k,j})\}_{j=1}^{n}$, respectively, and the set of cubature weights $\{(1/2)^{d}w_{k,j}\}_{j=1}^{n}$.  That is, the cubature nodes are versions of the original node set under a set of affine transformations and the weights are a scaled version of the original weight set.  The results of the following section show that these cubature rules are again exact for polynomials up to degree $m$ on the set subdomains.

First, however, it is important to realize that the integral of $f$ over $t_{k,0}$ can then be approximated as
\begin{align}
    \int\limits_{t_{k,0}}f(\mathbf{x})dV=&\sum_{i=1}^{2^{d}}\sum\limits_{j=1}^{n_{k}}(1/2)^{d}w_{k,j}f(T_{k,i}(\mathbf{x}_{k,j}))+\mathcal{E}_{t_{k,0},h_{k}/2}[f].\nonumber
\end{align}
Noting that
\begin{align}
    \int\limits_{t_{k,0}}f(\mathbf{x})dV=&\sum_{i=1}^{2^{d}}\int\limits_{t_{k,i}}f(\mathbf{x})dV\nonumber
\end{align}
the error $\mathcal{E}_{t_{k,0},h_{k}/2}[f]=\sum_{i=1}^{2^{d}}\mathcal{E}_{t_{k,i},h_{k}/2}[f]$.  Applying theorem \ref{thm:LocalErrorOrder} with $t_{k,i}$ in the place of $t_{k,0}$ and using the associated cubature rule, substitution of the Taylor formula for $f$ about $\mathbf{b}_{k}$ into this expression reveals
\begin{align}
\mathcal{E}_{t_{k,0},h_{k}/2}[f]=&\sum\limits_{i=1}^{2^{d}}\left(\sum_{\lvert\boldsymbol{\alpha}\rvert= m+1}\frac{1}{\boldsymbol{\alpha} !} \partial^{\boldsymbol{\alpha}}f(\mathbf{x})\big|_{\mathbf{x}=\mathbf{b}_{k}}\mathcal{E}_{t_{k,i},h_{k}/2}[\pi_{k,\boldsymbol{\alpha}}]+O(h_{k}^{m+2+d})\right).\label{eq:errorfhalf}
\end{align}
This expression motivates further interrogation of the error in approximating integrals of $d$-variate polynomials over each of the $d$-simplices in the subdivision.  Since each of $t_{k,i}$, $i=1,2,\ldots,d+1$, shares a vertex with $t_{k,0}$, these will be referred to as corner simplices in the subdivision.  On the other hand, $t_{k,i}$, $i=d+2,\ldots,2^d$, will be referred to as a central simplices of the subdivision.  Corner simplices and central simplices will be considered, in turn.

\subsubsection{Errors When Integrating Monomials Over Corner Simplices in the Subdivision} \label{sec:ErrorsCornerTriangles}

Notice that, for $i=1,2,\ldots,d+1$
\begin{align}
\int\limits_{t_{k,i}}(\mathbf{x}-\mathbf{b}_{k})^{\boldsymbol{\alpha}}dV = \left(\frac{1}{2}\right)^{d}\int\limits_{t_{k,0}}\left(\frac{1}{2}(\mathbf{x}+\mathbf{v}_{k,i})-\mathbf{b}_{k}\right)^{\boldsymbol{\alpha}}dV=\left(\frac{1}{2}\right)^{d+\lvert\boldsymbol{\alpha}\rvert}\int\limits_{t_{k,0}}\left((\mathbf{x}-\mathbf{b}_{k})+(\mathbf{v}_{k,i}-\mathbf{b}_{k})\right)^{\boldsymbol{\alpha}}dV\nonumber
\end{align}
and
\begin{align}
\sum\limits_{j=1}^{n_{k}}\left(\frac{1}{2}\right)^{d}w_{k,j}\left(\frac{1}{2}(\mathbf{x}_{k,j}+\mathbf{v}_{k,i})-\mathbf{b}_{k}\right)^{\boldsymbol{\alpha}}=\left(\frac{1}{2}\right)^{d+\lvert\boldsymbol{\alpha}\rvert}\sum\limits_{j=1}^{n_{k}}w_{k,j}\left((\mathbf{x}_{k,j}-\mathbf{b}_{k})+(\mathbf{v}_{k,i}-\mathbf{b}_{k})\right)^{\boldsymbol{\alpha}}\nonumber
\end{align}
Therefore,
\begin{align}
    \mathcal{E}_{t_{k,i},h_{k}/2}[\pi_{k,\boldsymbol{\alpha}}] = \left(\frac{1}{2}\right)^{d+\lvert\boldsymbol{\alpha}\rvert}\left(\int\limits_{t_{k,0}}\left((\mathbf{x}-\mathbf{b}_{k})+(\mathbf{v}_{k,i}-\mathbf{b}_{k})\right)^{\boldsymbol{\alpha}}dV-\sum\limits_{j=1}^{n_{k}}w_{k,j}\left((\mathbf{x}_{k,j}-\mathbf{b}_{k})+(\mathbf{v}_{k,i}-\mathbf{b}_{k})\right)^{\boldsymbol{\alpha}}\right).\label{eq:errortki}
\end{align}
The following lemma allows $\mathcal{E}_{t_{k,i},h_{k}/2}[\pi_{k,\boldsymbol{\alpha}}]$, $i=1,2,\ldots,d$, to be related to $\mathcal{E}_{t_{k,0},h_{k}}[\pi_{k,\boldsymbol{\alpha}}]$ in the proof of the theorem of the next section.  This fact is key to the superconvergence result of this work.
\begin{lemma}
Suppose that the cubature rule with weights $\{w_{k,j}\}_{j=1}^{n_{k}}$ and node set $\mathcal{X}_{k,0}$ is exact for polynomials up to degree $m$ and that $\lvert\boldsymbol{\alpha}\rvert=m+1$, then
\begin{align}
\int\limits_{t_{k,0}}\left((\mathbf{x}-\mathbf{b}_{k})+(\mathbf{v}_{k,i}-\mathbf{b}_{k})\right)^{\boldsymbol{\alpha}}dV-\sum\limits_{j=1}^{n_{k}}w_{k,j}\left((\mathbf{x}_{k,j}-\mathbf{b}_{k})+(\mathbf{v}_{k,i}-\mathbf{b}_{k})\right)^{\boldsymbol{\alpha}}=\mathcal{E}_{t_{k,0},h_{k}}[\pi_{k,\boldsymbol{\alpha}}].\nonumber
\end{align}
\end{lemma}
\begin{proof}
Application of the multi-binomial theorem produces
\begin{align}
\int\limits_{t_{k,0}}&\left((\mathbf{x}-\mathbf{b}_{k})+(\mathbf{v}_{k,i}-\mathbf{b}_{k})\right)^{\boldsymbol{\alpha}}-\sum\limits_{j=1}^{n_{k}}w_{k,j}\left((\mathbf{x}_{k,j}-\mathbf{b}_{k})+(\mathbf{v}_{k,i}-\mathbf{b}_{k})\right)^{\boldsymbol{\alpha}}=\nonumber\\
&\sum\limits_{\mathbf{0}\leq\boldsymbol{\beta}\leq\boldsymbol{\alpha}}\left(\begin{array}{c}\boldsymbol{\alpha} \\\boldsymbol{\beta}\end{array}\right)(\mathbf{v}_{k,i}-\mathbf{b}_{k})^{\boldsymbol{\beta}}\left(\int\limits_{t_{k,0}}\left(\mathbf{x}-\mathbf{b}_{k}\right)^{\boldsymbol{\alpha}-\boldsymbol{\beta}}dV-\sum\limits_{j=1}^{n_{k}}w_{k,j}\left(\mathbf{x}_{k,j}-\mathbf{b}_{k}\right)^{\boldsymbol{\alpha}-\boldsymbol{\beta}}\right).\nonumber
\end{align}
Since the cubature rule is exact for all polynomials up to degree $m$ and $\lvert\boldsymbol{\alpha}-\boldsymbol{\beta}\rvert\leq m$ when $\boldsymbol{\beta}\neq\mathbf{0}$ the result follows.
\end{proof}

\subsubsection{Errors When Integrating Monomials Over Central Simplices in the Subdivision} \label{sec:ErrorsCentralTriangles}

In the case of $d=2$ notice that
\begin{align}
\int\limits_{t_{k,4}}(\mathbf{x}-\mathbf{b}_{k})^{\boldsymbol{\alpha}}dV = \left(\frac{1}{2}\right)^{2}\iint\limits_{t_{k,0}}\left(-\frac{1}{2}(\mathbf{x}-\mathbf{b}_{k})+\mathbf{b}_{k}-\mathbf{b}_{k}\right)^{\boldsymbol{\alpha}}dV=\left(\frac{1}{2}\right)^{2+\lvert\boldsymbol{\alpha}\rvert}(-1)^{\lvert\boldsymbol{\alpha}\rvert}\int\limits_{t_{k,0}}\left(\mathbf{x}-\mathbf{b}_{k}\right)^{\boldsymbol{\alpha}}dV\nonumber
\end{align}
and
\begin{align}
\sum\limits_{j=1}^{n_{k}}\left(\frac{1}{2}\right)^{2}w_{k,j}\left(-\frac{1}{2}(\mathbf{x}_{k,j}-\mathbf{b}_{k})+\mathbf{b}_{k}-\mathbf{b}_{k}\right)^{\boldsymbol{\alpha}}=\left(\frac{1}{2}\right)^{2+\lvert\boldsymbol{\alpha}\rvert}(-1)^{\lvert\boldsymbol{\alpha}\rvert}\sum\limits_{j=1}^{n_{k}}w_{k,j}\left(\mathbf{x}_{k,j}-\mathbf{b}_{k}\right)^{\boldsymbol{\alpha}},\nonumber
\end{align}
so that
\begin{align}
    \mathcal{E}_{t_{k,4},h_{k}/2}[\pi_{k,\boldsymbol{\alpha}}] = \left(\frac{1}{2}\right)^{2+\lvert\boldsymbol{\alpha}\rvert}(-1)^{\lvert\boldsymbol{\alpha}\rvert}\mathcal{E}_{t_{k,0},h_{k}}[\pi_{k,\boldsymbol{\alpha}}].\nonumber
\end{align}
That is, the error when approximating the integral of a monomial over $t_{k,4}$ is a scaled version of the error when integrating over $t_{k,0}$, with a sign dependent on $\lvert\boldsymbol{\alpha}\rvert$.

Alternatively, for $d=3$
\begin{align}
\int\limits_{t_{k,5}}(\mathbf{x}-\mathbf{b}_{k})^{\boldsymbol{\alpha}}dV+\int\limits_{t_{k,8}}(\mathbf{x}-\mathbf{b}_{k})^{\boldsymbol{\alpha}}dV &= \left(\frac{1}{2}\right)^{3}\left(\int\limits_{t_{k,0}}\left(T_{k,5}(\mathbf{x})-\mathbf{b}_{k}\right)^{\boldsymbol{\alpha}}dV+\int\limits_{t_{k,0}}\left(T_{k,8}(\mathbf{x})-\mathbf{b}_{k}\right)^{\boldsymbol{\alpha}}dV\right)\nonumber\\
&=\left(\frac{1}{2}\right)^{3}(1+(-1)^{\lvert\boldsymbol{\alpha}\rvert})\int\limits_{t_{k,0}}\left(T_{k,5}(\mathbf{x})-\mathbf{b}_{k}\right)^{\boldsymbol{\alpha}}dV.\nonumber
\end{align}
and
\begin{align}
\sum\limits_{j=1}^{n_{k}}\left(\frac{1}{2}\right)^{3}w_{k,j}\left(T_{k,5}(\mathbf{x}_{k,j})-\mathbf{b}_{k}\right)^{\boldsymbol{\alpha}}+&\sum\limits_{j=1}^{n_{k}}\left(\frac{1}{2}\right)^{3}w_{k,j}\left(T_{k,8}(\mathbf{x}_{k,j})-\mathbf{b}_{k}\right)^{\boldsymbol{\alpha}}=\nonumber\\
&\left(\frac{1}{2}\right)^{3}(1+(-1)^{\lvert\boldsymbol{\alpha}\rvert})\sum\limits_{j=1}^{n_{k}}w_{k,j}\left(T_{k,5}(\mathbf{x}_{k,j})-\mathbf{b}_{k}\right)^{\boldsymbol{\alpha}}.\nonumber
\end{align}
Construction of the Taylor formula of $(EPE^{-1}(\mathbf{x}-(1/2)(\mathbf{v}_{k,1}+\mathbf{v}_{k,2})))^{\boldsymbol{\alpha}}$ to order $\lvert\boldsymbol{\alpha}\rvert$ about $\mathbf{x}=(1/2)(\mathbf{v}_{k,1}+\mathbf{v}_{k,2})$ followed by application of the multinomial theorem \cite[Section 26.4(ii)]{NISTDLMF} reveals that
\begin{align}
\left(T_{k,5}(\mathbf{x})-\mathbf{b}_{k}\right)^{\boldsymbol{\alpha}}=\sum\limits_{\lvert\boldsymbol{\beta}\rvert=\lvert\boldsymbol{\alpha}\rvert}c_{\boldsymbol{\alpha},\boldsymbol{\beta}}(EPE^{-1})(\mathbf{x}-(1/2)(\mathbf{v}_{k,1}+\mathbf{v}_{k,2}))^{\boldsymbol{\beta}}\nonumber
\end{align}
with coefficients $c_{\boldsymbol{\alpha},\boldsymbol{\beta}}(EPE^{-1})$ that are independent of $h_{k}$.  Also, from \cite[Section 4]{JAR2026}
\begin{align}
\int\limits_{t_{k,0}}\left(\mathbf{x}-(1/2)(\mathbf{v}_{k,1}+\mathbf{v}_{k,2})\right)^{\boldsymbol{\beta}}dV-\sum\limits_{j=1}^{n_{k}}w_{k,j}\left(\mathbf{x}_{k,j}-(1/2)(\mathbf{v}_{k,1}+\mathbf{v}_{k,2})\right)^{\boldsymbol{\beta}}=O(h_{k}^{\lvert\boldsymbol{\beta}\rvert+3}).\nonumber
\end{align}
Similar results hold when summing the terms involving $t_{k,6}$ and $t_{k,7}$, so that
\begin{align}
\sum_{i=5}^{8}\mathcal{E}_{t_{k,i},h_{k}/2}[\pi_{k,\boldsymbol{\alpha}}]=\left(\frac{1}{2}\right)^{3}(1+(-1)^{\lvert\boldsymbol{\alpha}\rvert})O(h_{k}^{\lvert\boldsymbol{\alpha}\rvert+3})
\end{align}
This sum vanishes when $\boldsymbol{\alpha}$ has even order.

\section{Superconvergence on a Polytopal Domain} \label{sec:TotalError}

The results of sections \ref{sec:ErrorsCornerTriangles} and \ref{sec:ErrorsCentralTriangles} can be combined into a superconvergence result over $\Omega$ by comparing the error in the composite cubature over $\Omega$ before and after subdivision. The following theorem \ref{thm:LocalConvergenceTheorem} states this result.
\begin{theorem} \label{thm:LocalConvergenceTheorem}
Suppose that for each $k=1,2,\ldots,K$ the assumptions of theorem \ref{thm:LocalErrorOrder} are satisfied and that $K=O(h^{-d})$ as $h\to0$. Then
\begin{align}
    \mathcal{E}_{\Omega,h}[f]=\left\{\begin{array}{cc}O(h^{m+2}),& m\mbox{ even}\\
    O(h^{m+1}), & m\mbox{ odd}\end{array}\right.\nonumber
\end{align}
as $h\to0$.
\end{theorem}

\begin{proof}
Applying the lemma to \eqref{eq:errortki} reveals that, for $i=1,2,\ldots,d+1$
\begin{align}
    \mathcal{E}_{t_{k,i},h_{k}/2}[\pi_{k,\boldsymbol{\alpha}}] = \left(\frac{1}{2}\right)^{d+\lvert\boldsymbol{\alpha}\rvert}\mathcal{E}_{t_{k,0},h_{k}}[\pi_{k,\boldsymbol{\alpha}}].\nonumber
\end{align}

For the case of $d=2$, returning to the expression \eqref{eq:errorfhalf}, the results of section \ref{sec:ErrorsCornerTriangles} and \ref{sec:ErrorsCentralTriangles} imply, as $h\to0$,
\begin{align}
\mathcal{E}_{t_{k,0},h_{k}/2}[f]=&\sum_{\lvert\boldsymbol{\alpha}\rvert= m+1}\frac{1}{\boldsymbol{\alpha} !} \partial^{\boldsymbol{\alpha}}f(\mathbf{x})\big|_{\mathbf{x}=\mathbf{b}_{k}}\left(\frac{1}{2}\right)^{2+\lvert\boldsymbol{\alpha}\rvert}\left(3+(-1)^{\lvert\boldsymbol{\alpha}\rvert}\right)\mathcal{E}_{t_{k,0},h_{k}}[\pi_{k,\boldsymbol{\alpha}}]+O(h_{k}^{m+4})\nonumber\\
=&\left(\frac{1}{2}\right)^{m+3-(1+(1/2)(1-(-1)^{m}))}\sum_{\lvert\boldsymbol{\alpha}\rvert= m+1}\frac{1}{\boldsymbol{\alpha} !} \partial^{\boldsymbol{\alpha}}f(\mathbf{x})\big|_{\mathbf{x}=\mathbf{b}_{k}}\mathcal{E}_{t_{k,0},h_{k}}[\pi_{k,\boldsymbol{\alpha}}]+O(h_{k}^{m+4}),\nonumber
\end{align}
Therefore,
\begin{align}
\mathcal{E}_{\Omega,h/2}[f]=&\left(\frac{1}{2}\right)^{m+3-(1+(1/2)(1-(-1)^{m}))}\sum_{k=1}^{K}\left(\sum_{\lvert\boldsymbol{\alpha}\rvert= m+1}\frac{1}{\boldsymbol{\alpha} !} \partial^{\boldsymbol{\alpha}}f(\mathbf{x})\big|_{\mathbf{x}=\mathbf{b}_{k}}\mathcal{E}_{t_{k,0},h_{k}}[\pi_{k,\boldsymbol{\alpha}}]+O(h^{m+4})\right).\nonumber
\end{align}
Further, noting that, as $h\to0$, $\sum_{k=1}^{K}O(h^{m+4})=O(h^{m+2})$ (under the assumption that $K=O(h^{-d})$)
\begin{align}
\frac{\mathcal{E}_{\Omega,h/2}[f]}{\mathcal{E}_{\Omega,h}[f]}=\left(\frac{1}{2}\right)^{m+3-(1+(1/2)(1-(-1)^{m}))}+O(h^{m+2}).\label{eq:errorratio}
\end{align}
To recover the convergence order $\rho$, express $\mathcal{E}_{\Omega,h}[f]$ as a power series in $h$ so that as $h\to0$
\begin{align}
\mathcal{E}_{\Omega,h}[f]= c h^{\rho}+o(h^{\rho})=c h^{\rho}\left(1+\frac{o(1)}{c}\right)=c h^{\rho}\left(1+o(1)\right). \nonumber
\end{align}
Therefore, $\mathcal{E}_{\Omega,h/2}[f]/\mathcal{E}_{\Omega,h}[f]= (1/2)^{\rho}(1+o(1))$ as $h\to0$.  Comparing this to \eqref{eq:errorratio} it is clear that as $h\to 0$, $\rho = m+2$, when $m$ is even, otherwise $\rho=m+1$.

On the other hand, for the case of $d=3$, returning again to the expression \eqref{eq:errorfhalf}, the results of section \ref{sec:ErrorsCentralTriangles} show that for even $m$ as $h\to0$
\begin{align}
\mathcal{E}_{t_{k,0},h_{k}/2}[f]=&\sum_{\lvert\boldsymbol{\alpha}\rvert= m+1}\frac{1}{\boldsymbol{\alpha} !} \partial^{\boldsymbol{\alpha}}f(\mathbf{x})\big|_{\mathbf{x}=\mathbf{b}_{k}}\left(4\left(\frac{1}{2}\right)^{3+\lvert\boldsymbol{\alpha}\rvert}\mathcal{E}_{t_{k,0},h_{k}}[\pi_{k,\boldsymbol{\alpha}}]\right)+O(h_{k}^{m+5})\nonumber\\
=&\left(\frac{1}{2}\right)^{m+2}\sum_{\lvert\boldsymbol{\alpha}\rvert= m+1}\frac{1}{\boldsymbol{\alpha} !} \partial^{\boldsymbol{\alpha}}f(\mathbf{x})\big|_{\mathbf{x}=\mathbf{b}_{k}}\mathcal{E}_{t_{k,0},h_{k}}[\pi_{k,\boldsymbol{\alpha}}]+O(h_{k}^{m+5}),\nonumber
\end{align}
so that
\begin{align}
\mathcal{E}_{\Omega,h/2}[f]=&\left(\frac{1}{2}\right)^{m+2}\sum_{k=1}^{K}\left(\sum_{\lvert\boldsymbol{\alpha}\rvert= m+1}\frac{1}{\boldsymbol{\alpha} !} \partial^{\boldsymbol{\alpha}}f(\mathbf{x})\big|_{\mathbf{x}=\mathbf{b}_{k}}\mathcal{E}_{t_{k,0},h_{k}}[\pi_{k,\boldsymbol{\alpha}}]+O(h^{m+5})\right).\nonumber
\end{align}
Recovering the convergence order as in the case of $d=2$ for even $m$ reveals that again $\mathcal{E}_{\Omega,h}[f]=O(h^{m+2})$ as $h\to0$ in this case.  When $m$ is odd, the discussion following theorem \ref{thm:LocalErrorOrder} reveals that $\mathcal{E}_{\Omega,h}[f]=O(h^{m+1})$ as $h\to0$.
\end{proof}

\section{Numerical Results} \label{sec:Numerics}

To demonstrate the superconvergence results developed here, results of numerical experiments are presented for each of $d=2$ and $d=3$, in turn.

\subsection{Experimental Results When $d=2$}
Consider approximating the integrals of
\begin{align}
f_{1}(\mathbf{x}) = \sin(\pi\lVert\mathbf{x}-\mathbf{x}_{0}\rVert_{2}^{2})\nonumber
\end{align}
and
\begin{align}
f_{2}(\mathbf{x}) = \lVert\mathbf{x}-\mathbf{x}_{0}\rVert_{2}^{5},\nonumber
\end{align}
with $\mathbf{x}_{0}=[-0.09\mbox{ }0.39]^{T}$ a randomly chosen shift, over the unit square centered at the origin (i.e., $\Omega=[-1/2,1/2]^{2}$, so that $\int_{\Omega}f_{1}(\mathbf{x})dV\approx0.591327220397927$  and $\int_{\Omega}f_{2}(\mathbf{x})dV\approx0.134574061691459$). Approximations are constructed on 10 distinct quasi-uniformly, pseudo-randomly and randomly spaced node sets, $\mathcal{X}_{N}$, for each of 30 values of $N$ that are nearly equally spaced by base 10 logarithm ranging from $10^{3}$ to $10^{5}$ (generated via {\verb+ceil(logspace(3,5,30))+} in Matlab).  Illustrations of examples of these node sets for $N=1000$ and $N=1888$ along with the relationships between $h$ and $K$ with $N$ are displayed in figure \ref{fig:ExampleTriangulations}.  In all cases, it is roughly true that $K=O(h^{-2})$, with a slight deviation for the node sets drawn from the random uniform distribution.  The sets $\{t_{k,0}\}_{k=1}^{K}$ include the simplices defined by the Delaunay triangulation.
\begin{figure}
\begin{center}
\includegraphics[width=\linewidth]{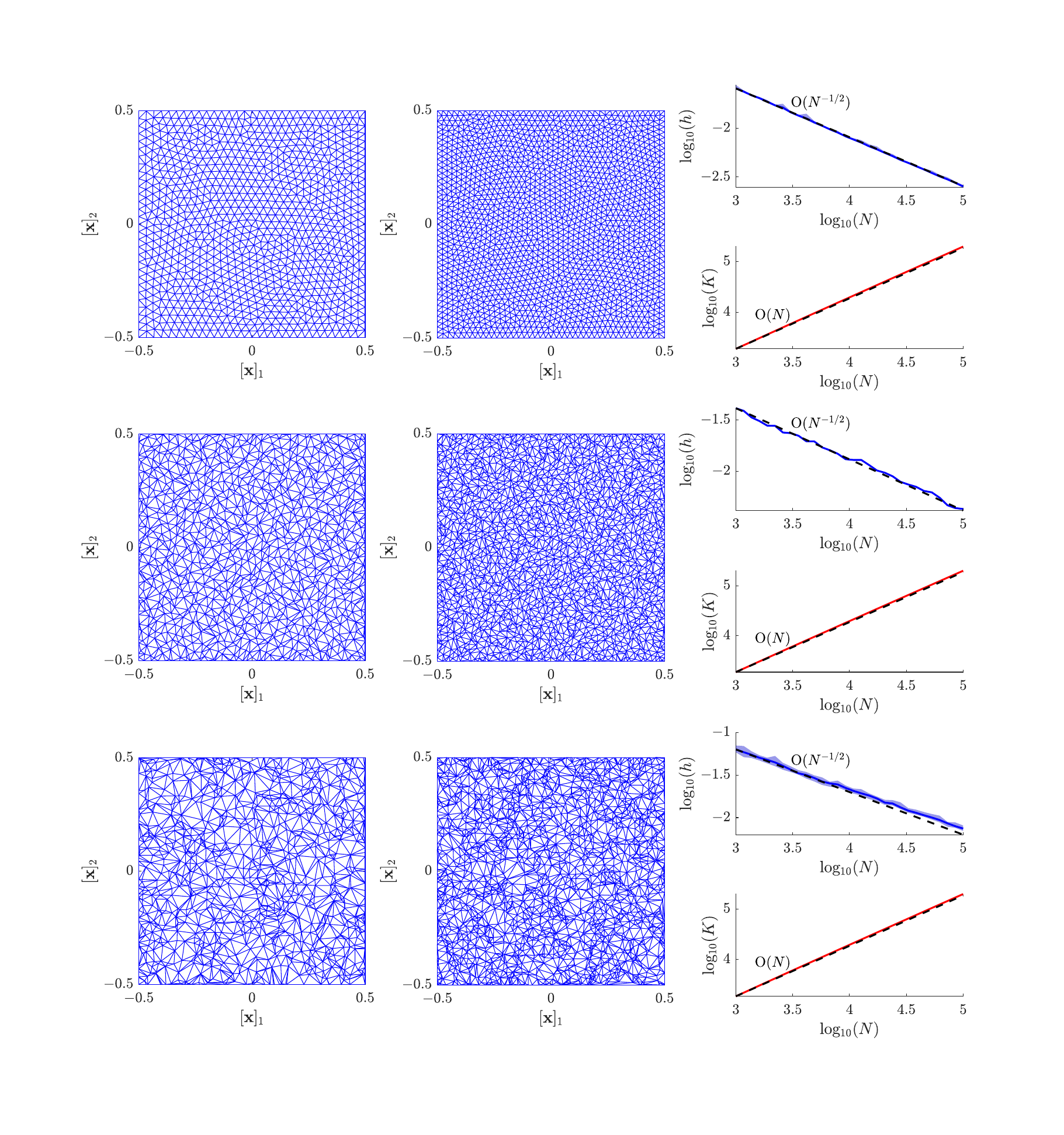}
\end{center}
\caption{Examples of node sets and triangulations used in the computational examples.  In the first two columns, nodes in $\mathcal{X}_{N}$ are the vertices of the triangles, while the edges of the triangles are the depicted line segments.  The first column has $N=1000$ while the second has $N=1888$.  The third column illustrates the relationships between $h$ and $K$ with $N$. For each of 30 values of $N$ that are nearly equally spaced by base 10 logarithm ranging from $10^{3}$ to $10^{5}$, 10 distinct node sets and triangulations have been generated. The regions bounded by the maximum and minimum values of $h$ and $K$, with respect to each value of $N$, are shaded with the most variability in $h$ in the third row.  The mean value of $h$ and $K$ are shown as solid curves, with the maximum and minimum values indistinguishable from the mean in many cases.  Reference lines (dashed and labeled) at the appropriate order are shown.  The first row illustrates quasi-uniformly spaced node sets. The second row illustrates node sets where the interior nodes are drawn from the Halton set in $\mathbb{R}^{2}$ and the nodes on each edge of the boundary are spaced as drawn from the Halton set in $\mathbb{R}$.  The third row illustrates node sets where the interior nodes are drawn from the uniform distribution on $(-1/2,1/2)^{2}$ and the nodes on each edge of the boundary are spaced as drawn from the uniform distribution on $[-1/2,1/2]$.}
\label{fig:ExampleTriangulations}
\end{figure}

Local cubature weight sets are computed using the two procedures described in section \ref{sec:WeightSets}, with $n_{k}=2M_{d,m}$, a choice guided by observations in, e.g., \cite{JARBF2017}, and $\varphi(r) = r$.  The local cubature node sets, $\mathcal{X}_{k,0}$, consist of the $n_{k}$ nodes in $\mathcal{X}_{N}$ nearest to $\mathbf{b}_{k}$.  Note that all computations presented here were performed on a workstation with two Intel$^\circledR$ Xeon$^\circledR$ CPU E5-2697 v3 processors, each running at 2.60GHz, and 256 GB of memory running MATLAB R2022b.  For each value of $m=1,2,\ldots,7$, figures \ref{fig:Errorf1} and \ref{fig:Errorf2} illustrate the total error, $\mathcal{E}_{\Omega,h}[f_{i}]$, $i=1,2$, respectively, where the error shown is the maximum over the 10 distinct node sets for each choice of $N$.  Figure \ref{fig:ExampleTriangulations} indicates that for these sets of nodes $N=O(h^{-2})$, so that the error is expected to behave as
\begin{align}
    \mathcal{E}_{\Omega,h}[f]=\left\{\begin{array}{cc}O(N^{-(m+2)/2}),& m\mbox{ even}\\
    O(N^{-(m+1)/2}), & m\mbox{ odd}\end{array}\right..\nonumber
\end{align}
The figures both demonstrate that these orders of convergence achieved except for larger values of $N$ when $m=1,2$ using nodes drawn from the uniform distribution.  The appropriate orders of convergence can be restored by greater oversampling, here choosing $n_{k}=80$ in all cases, as depicted in figures \ref{fig:Errorf1Oversample} and \ref{fig:Errorf2Oversample}.  Further, the superconvergence result of section \ref{sec:TotalError} relies on the integrand having enough continuous mixed partial derivatives.  The integrand $f_{2}$ has continuous mixed partial derivatives up to order $5$, so that achieving $O(h^{-8})$ ($O(N^{-8/2})$) convergence should not be expected.  Here, the convergence order is appropriately limited to $O(h^{-7})$ ($O(N^{-7/2})$).

\begin{figure}
\begin{center}
\includegraphics[width=\linewidth]{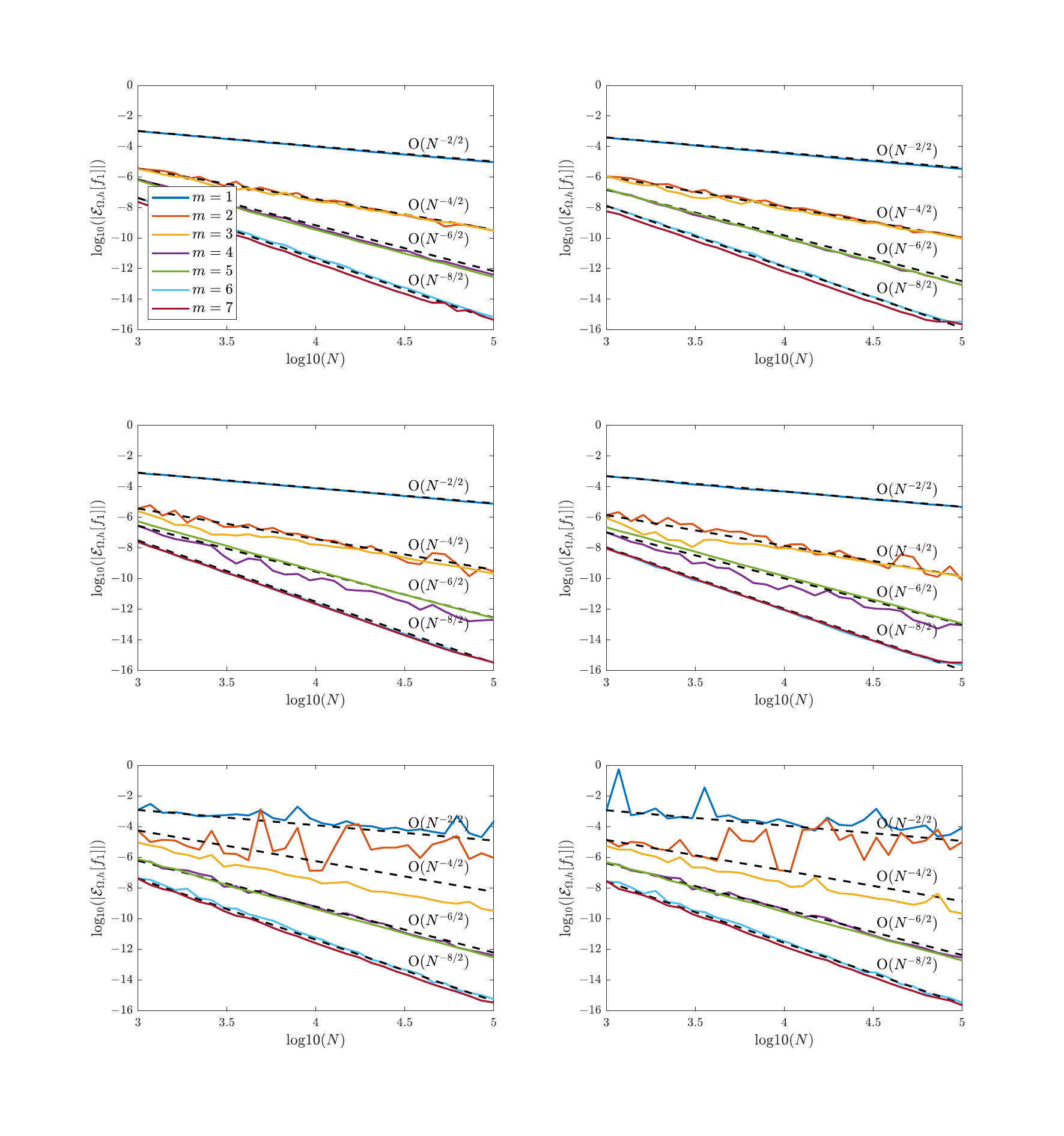}
\end{center}
\caption{Log base 10 of the absolute error when approximating the definite integral of $f_{1}$ over $[-1/2,1/2]^{2}$ (with $n_{k}=(m+1)(m+2)$).  For each value of $N$, the error shown is the maximum over ten distinct node sets/triangulations.  Frames in the left column are the result of generating cubature weights by minimizing the 2-norm squared of the weight vector $\mathbf{w}_{k}$ subject to \eqref{eq:WeightSystem}, while the weights used in the right column minimize \eqref{eq:RBFObjective}.  Results in the top row of frames are generated for quasi-uniform node sets.  The results of the middle row are for pseudo-random nodes taken from the shifted Halton set in $\mathbb{R}^{2}$ on the interior of the domain and with nodes on the boundary of the domain spaced as drawn from the shifted Halton node set in $\mathbb{R}$. The results of the bottom row are for random nodes taken from the uniform distribution on $(-1/2,1/2)^{2}$ on the interior of the domain and with nodes on the boundary of the domain spaced as drawn from the uniform distribution on $[-1/2,1/2]$.  Dashed lines are included and labeled for reference at the expected order of convergence.}
\label{fig:Errorf1}
\end{figure}

\begin{figure}
\begin{center}
\includegraphics[width=\linewidth]{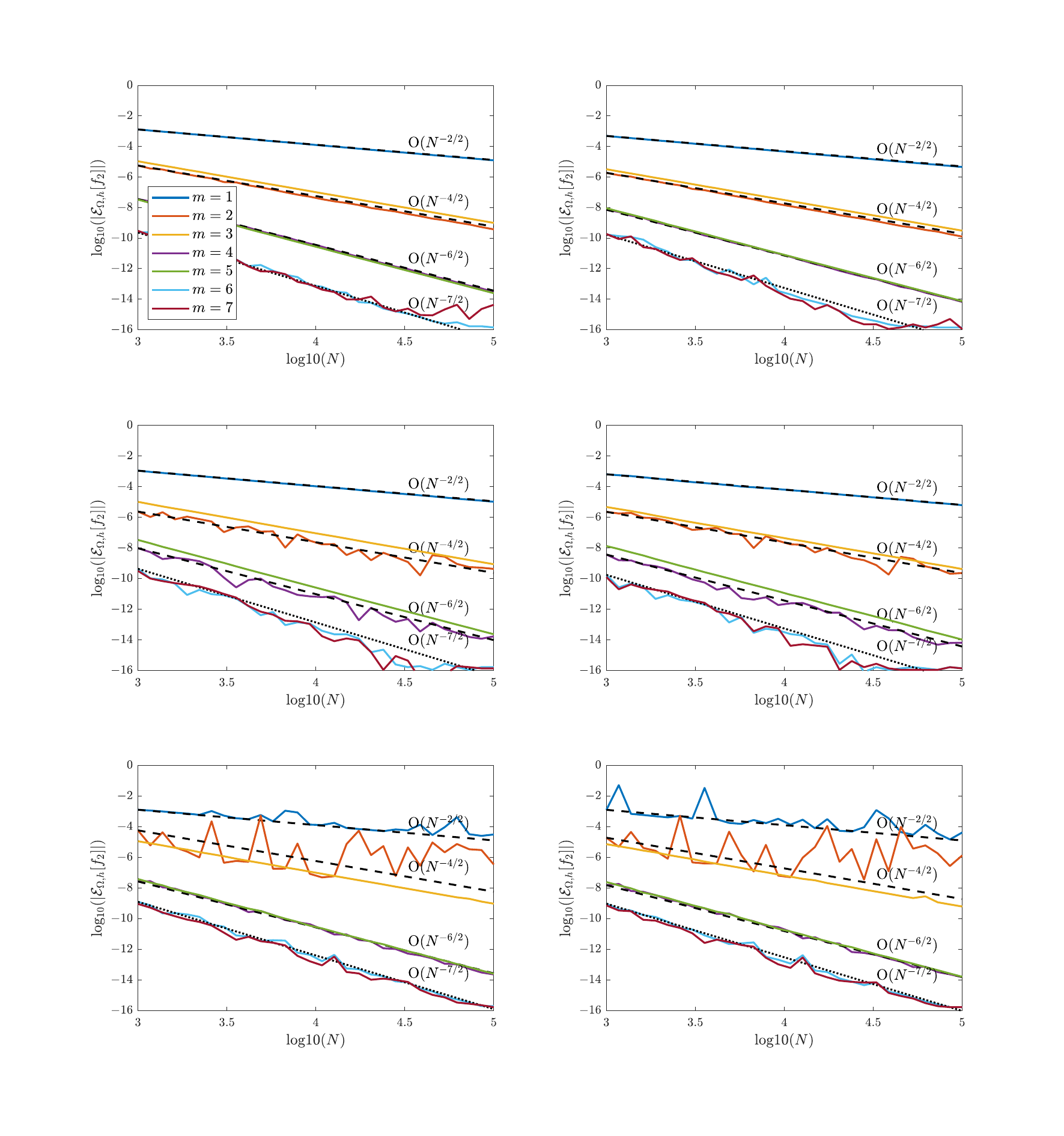}
\end{center}
\caption{Log base 10 of the absolute error when approximating the definite integral of $f_{1}$ over $[-1/2,1/2]^{2}$  (with $n_{k}=(m+1)(m+2)$).  For each value of $N$, the error shown is the maximum over ten distinct node sets/triangulations.  Frames in the left column are the result of generating cubature weights by minimizing the 2-norm squared of the weight vector $\mathbf{w}_{k}$ subject to \eqref{eq:WeightSystem}, while the weights used in the right column minimize \eqref{eq:RBFObjective}.  Results in the top row of frames are generated for quasi-uniform node sets.  The results of the middle row are for pseudo-random nodes taken from the shifted Halton set in $\mathbb{R}^{2}$ on the interior of the domain and with nodes on the boundary of the domain spaced as drawn from the shifted Halton node set in $\mathbb{R}$. The results of the bottom row are for random nodes taken from the uniform distribution on $(-1/2,1/2)^{2}$ on the interior of the domain and with nodes on the boundary of the domain spaced as drawn from the uniform distribution on $[-1/2,1/2]$.  Dashed lines are included and labeled for reference at the expected order of convergence.}
\label{fig:Errorf2}
\end{figure}

\begin{figure}
\begin{center}
\includegraphics[width=\linewidth]{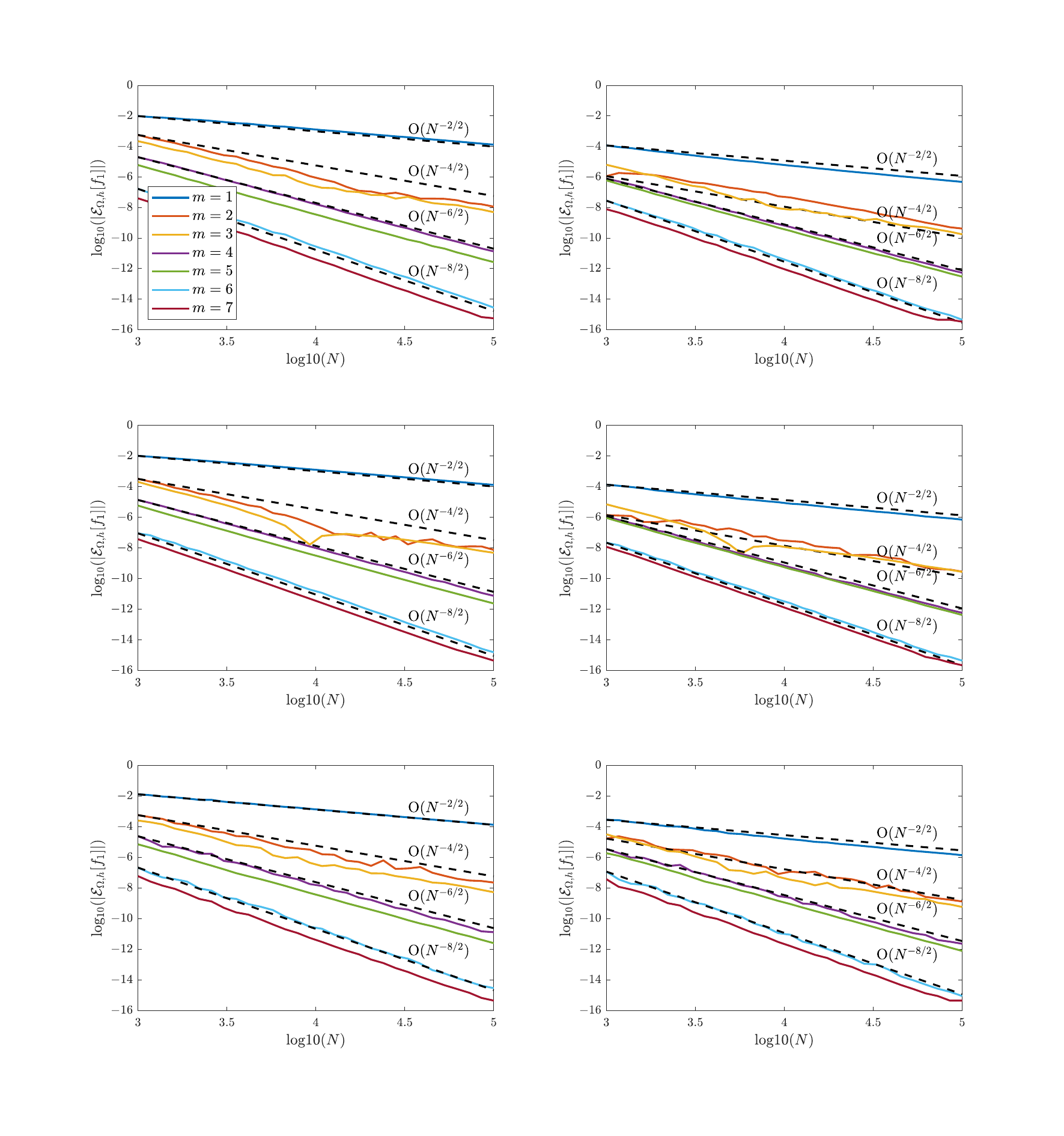}
\end{center}
\caption{Log base 10 of the absolute error when approximating the definite integral of $f_{1}$ over $[-1/2,1/2]^{2}$  (with $n_{k}=80$).  For each value of $N$, the error shown is the maximum over ten distinct node sets/triangulations.  Frames in the left column are the result of generating cubature weights by minimizing the 2-norm squared of the weight vector $\mathbf{w}_{k}$ subject to \eqref{eq:WeightSystem}, while the weights used in the right column minimize \eqref{eq:RBFObjective}.  Results in the top row of frames are generated for quasi-uniform node sets.  The results of the middle row are for pseudo-random nodes taken from the shifted Halton set in $\mathbb{R}^{2}$ on the interior of the domain and with nodes on the boundary of the domain spaced as drawn from the shifted Halton node set in $\mathbb{R}$. The results of the bottom row are for random nodes taken from the uniform distribution on $(-1/2,1/2)^{2}$ on the interior of the domain and with nodes on the boundary of the domain spaced as drawn from the uniform distribution on $[-1/2,1/2]$.  Dashed lines are included and labeled for reference at the expected order of convergence.}
\label{fig:Errorf1Oversample}
\end{figure}

\begin{figure}
\begin{center}
\includegraphics[width=\linewidth]{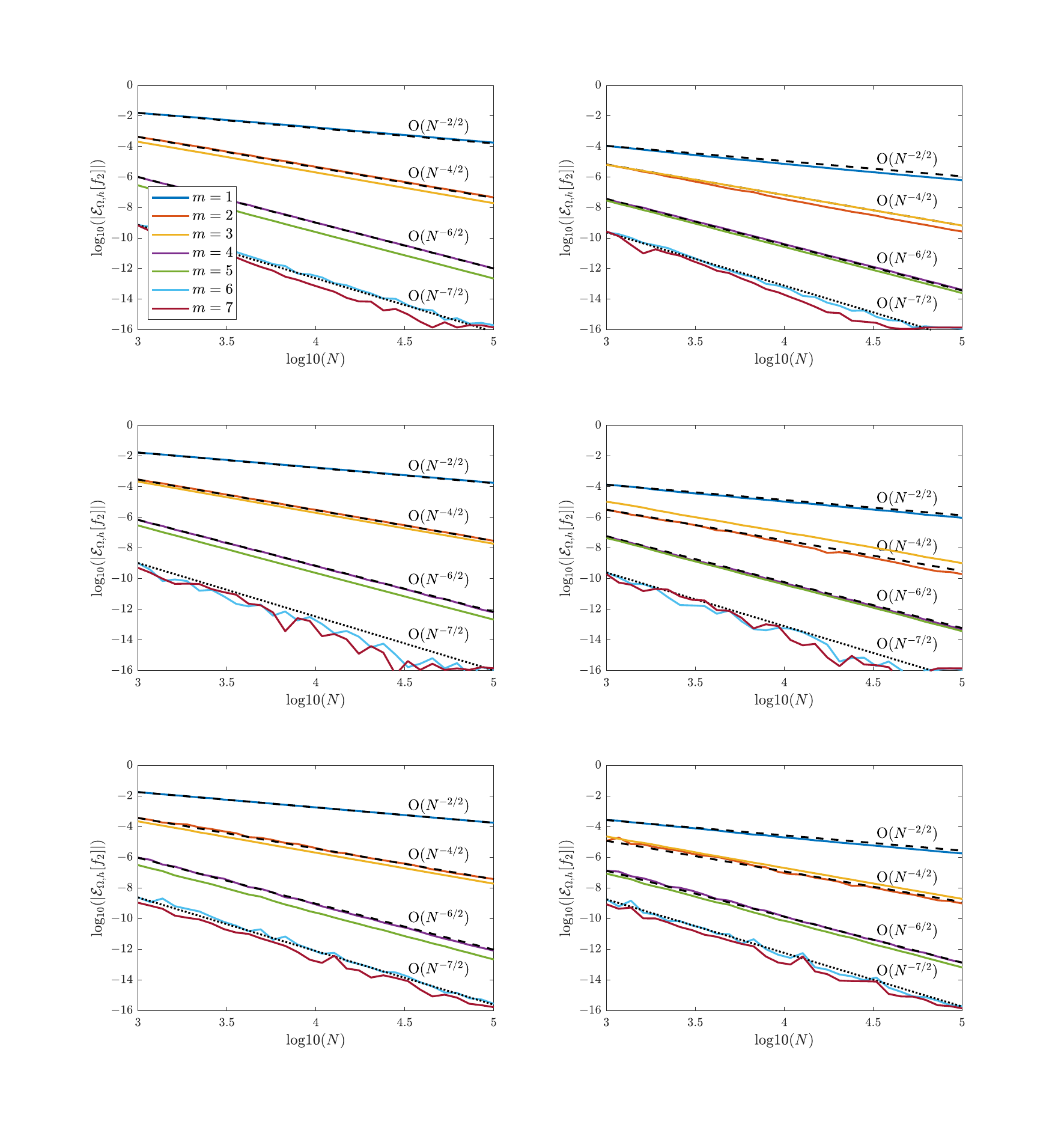}
\end{center}
\caption{Log base 10 of the absolute error when approximating the definite integral of $f_{1}$ over $[-1/2,1/2]^{2}$ (with $n_{k}=80$).  For each value of $N$, the error shown is the maximum over ten distinct node sets/triangulations.  Frames in the left column are the result of generating cubature weights by minimizing the 2-norm squared of the weight vector $\mathbf{w}_{k}$ subject to \eqref{eq:WeightSystem}, while the weights used in the right column minimize \eqref{eq:RBFObjective}.  Results in the top row of frames are generated for quasi-uniform node sets.  The results of the middle row are for pseudo-random nodes taken from the shifted Halton set in $\mathbb{R}^{2}$ on the interior of the domain and with nodes on the boundary of the domain spaced as drawn from the shifted Halton node set in $\mathbb{R}$. The results of the bottom row are for random nodes taken from the uniform distribution on $(-1/2,1/2)^{2}$ on the interior of the domain and with nodes on the boundary of the domain spaced as drawn from the uniform distribution on $[-1/2,1/2]$.  Dashed lines are included and labeled for reference at the expected order of convergence.}
\label{fig:Errorf2Oversample}
\end{figure}

\subsection{Experimental Results When $d=3$}

For $d=3$ numerical experiments are conducted on $f_{2}$ (now with $\mathbf{x}_{0}=[-0.09\mbox{ }0.39\mbox{ }0.31]^{T}$) and for $\Omega=[-1/2,1/2]^{3}$ (so that $\int_{\Omega}f_{2}(\mathbf{x})dV\approx0.323767286425822$).  Approximations are constructed on 10 distinct quasi-uniformly spaced node sets, $\mathcal{X}_{N}$, for each of 20 values of $N$ that are nearly equally spaced by base 10 logarithm ranging from $10^{3}$ to $23951$ (generated via taking the first 20 entries of {\verb+ceil(logspace(3,5,30))+} in Matlab).  Illustrations of the relationships between $h$ and $K$ with $N$ are displayed in figure \ref{fig:NodeSets3D}.  It is clear that $K$ behaves as roughly $O(h^{-3})$.  The sets $\{t_{k,0}\}_{k=1}^{K}$ include the simplices defined by the Delaunay tesselation of the node set.

\begin{figure}
\begin{center}
\includegraphics[width=\linewidth]{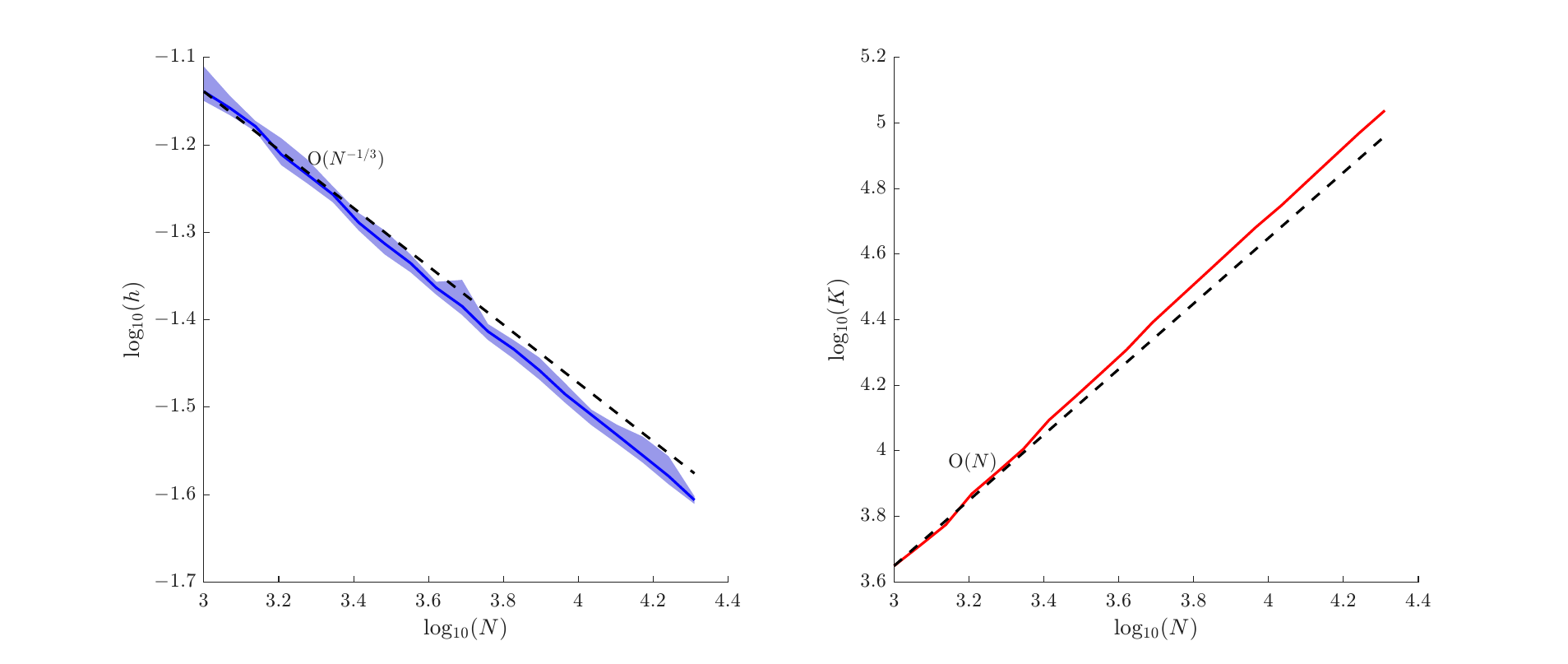}
\end{center}
\caption{Relationships between $h$ and $K$ with $N$. For each of 20 values of $N$ that are nearly equally spaced by base 10 logarithm ranging from $10^{3}$ to $23951$, 10 distinct quasi-uniformly spaced node sets and tesselations on $[-1/2,1/2]^{3}$ have been generated. The regions bounded by the maximum and minimum values of $h$ and $K$, with respect to each value of $N$, are shaded.  The mean value of $h$ and $K$ are shown as solid curves, with the maximum and minimum values indistinguishable from the mean in many cases.  Reference lines (dashed and labeled) at the appropriate order are shown.}
\label{fig:NodeSets3D}
\end{figure}

Local cubature weight sets are computed using the two procedures described in section \ref{sec:WeightSets}, with $n_{k}=2M_{d,m}$ and $\varphi(r) = r^3$.  The local cubature node sets, $\mathcal{X}_{k,0}$, consist of the $n_{k}$ nodes in $\mathcal{X}_{N}$ nearest to $\mathbf{b}_{k}$.  Note that all computations presented here were performed on a workstation with two Intel$^\circledR$ Xeon$^\circledR$ CPU E5-2697 v3 processors, each running at 2.60GHz, and 256 GB of memory running MATLAB R2022b.  For each value of $m=1,2,\ldots,7$, figure \ref{fig:Errorf23D} illustrates the total error, $\mathcal{E}_{\Omega,h}[f_{2}]$, where the error shown is the maximum over the 10 distinct node sets for each choice of $N$.  Figure \ref{fig:NodeSets3D} indicates that for these sets of nodes $N=O(h^{-3})$, so that the error is expected to behave as
\begin{align}
    \mathcal{E}_{\Omega,h}[f]=\left\{\begin{array}{cc}O(N^{-(m+2)/3}),& m\mbox{ even}\\
    O(N^{-(m+1)/3}), & m\mbox{ odd}\end{array}\right..\nonumber
\end{align}
The figure demonstrate that these orders of convergence achieved.  The superconvergence result of section \ref{sec:TotalError} again relies on the integrand having enough continuous mixed partial derivatives.  The integrand $f_{2}$ has continuous mixed partial derivatives up to order $5$, so that achieving $O(h^{-8})$ ($O(N^{-8/3})$) convergence should not be expected.  Here, the convergence order is appropriately limited to $O(h^{-7})$ ($O(N^{-7/2})$).

\begin{figure}
\begin{center}
\includegraphics[width=\linewidth]{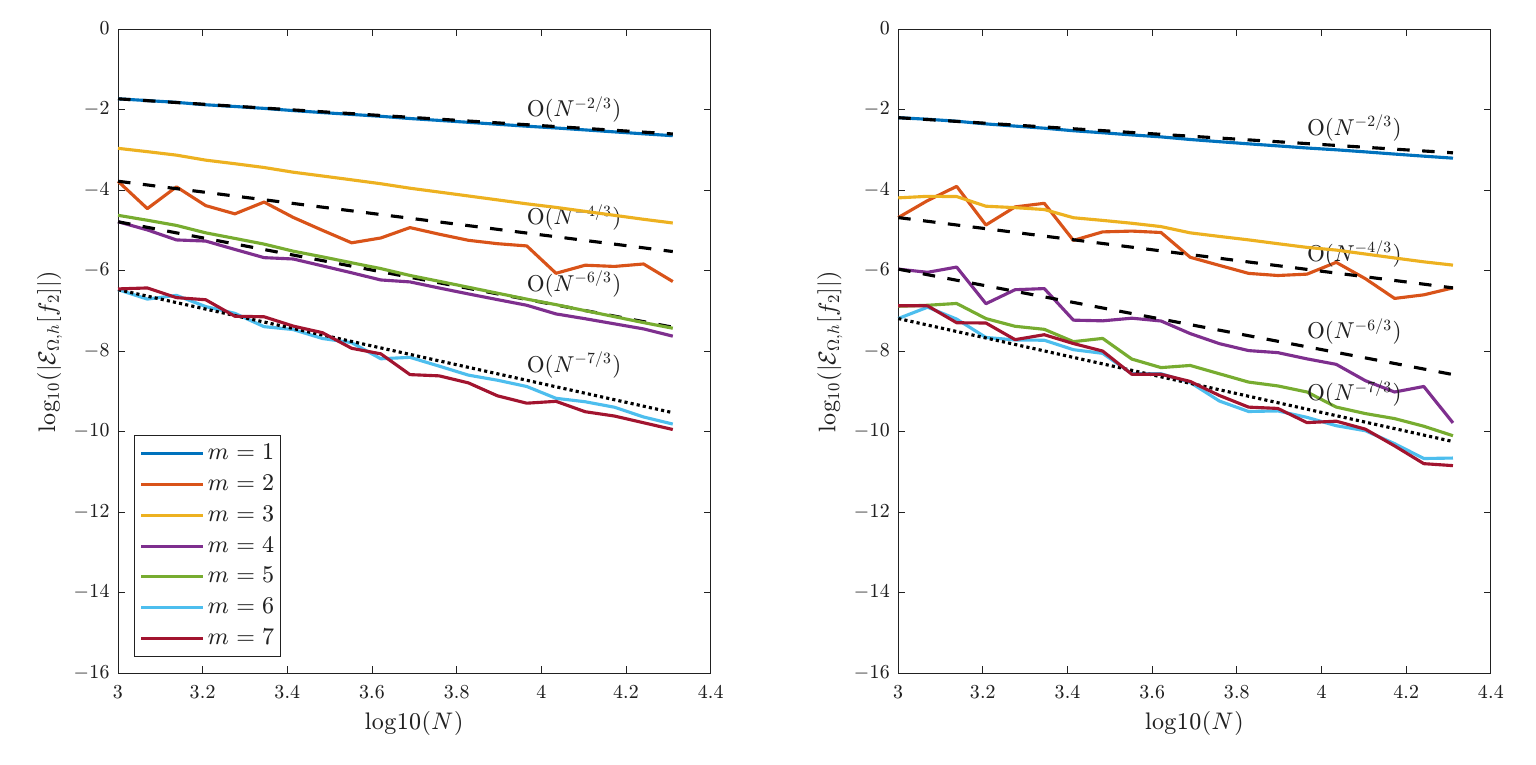}%Errorf13D does include the result for f2...it was cropped.
\end{center}
\caption{Log base 10 of the absolute error when approximating the definite integral of $f_{2}$ over $[-1/2,1/2]^{3}$ (with $n_{k}=(m+1)(m+2)(m+3)/3$).  For each value of $N$, the error shown is the maximum over ten distinct node sets/triangulations.  Frames in the left column are the result of generating cubature weights by minimizing the 2-norm squared of the weight vector $\mathbf{w}_{k}$ subject to \eqref{eq:WeightSystem}, while the weights used in the right column minimize \eqref{eq:RBFObjective}.  All results are generated for quasi-uniform node sets. Dashed lines are included and labeled for reference at the expected order of convergence.}
\label{fig:Errorf23D}
\end{figure}

\section{Conclusions} \label{sec:Conclusions}

A superconvergence result for approximating definite integrals over polytopal domains in dimension 2 and 3 has been shown.  This important result is intended to guide the choice of polynomial order for which a cubature rule is made exact.  Given the relationship between $n_{k}$ (the number of cubature nodes/weights) and $m$ (that is, $n_{k}\geq M_{d,m}$) necessary for the system of linear equations $P_{k}^{T}\mathbf{w}_{k}=\boldsymbol{\pi}_{k}$ to be guaranteed a solution when $\mathcal{X}_{k,0}$ is $\mathbb{P}_{m}^{2}$ unisolvent, it is important to keep $m$ as small as possible.  This result demonstrates that the choice of even $m$ should be a primary consideration when constructing cubature rules over polytopal domains that rely on approximating definite integrals over $d$-simplex subdomains.

\bibliographystyle{siamplain}
\bibliography{references}

\end{document}